\documentclass{lmcs}

\usepackage{amssymb}
\usepackage{array}
\usepackage{booktabs}
\usepackage{hyperref}
\usepackage[nameinlink,noabbrev]{cleveref}
\crefname{thm}{theorem}{theorems}
\Crefname{thm}{Theorem}{Theorems}
\crefname{lem}{lemma}{lemmas}
\Crefname{lem}{Lemma}{Lemmas}
\crefname{prop}{proposition}{propositions}
\Crefname{prop}{Proposition}{Propositions}
\crefname{cor}{corollary}{corollaries}
\Crefname{cor}{Corollary}{Corollaries}
\crefname{defi}{definition}{definitions}
\Crefname{defi}{Definition}{Definitions}
\crefname{exa}{example}{examples}
\Crefname{exa}{Example}{Examples}
\crefname{rem}{remark}{remarks}
\Crefname{rem}{Remark}{Remarks}

\newcommand{\FL}{\mathsf{FL}}
\newcommand{\EPFL}{\exists^{+}\!\mathsf{FL}}
\newcommand{\EFL}{\exists\mathsf{FL}}
\newcommand{\qr}{\operatorname{qr}}
\newcommand{\At}{\operatorname{At}}
\newcommand{\at}{\operatorname{at}}
\newcommand{\tp}{\operatorname{tp}}
\newcommand{\dom}{\operatorname{dom}}
\newcommand{\true}{\top}
\newcommand{\false}{\bot}
\newcommand{\A}{\mathfrak{A}}
\newcommand{\B}{\mathfrak{B}}
\newcommand{\C}{\mathfrak{C}}
\newcommand{\D}{\mathfrak{D}}
\newcommand{\Zt}{\mathsf{Z}}
\newcommand{\Ot}{\mathsf{O}}
\newcommand{\Mt}{\mathsf{M}}
\newcommand{\Types}{\{\Zt,\Ot,\Mt\}}
\newcommand{\prof}{\operatorname{prof}}
\newcolumntype{L}[1]{>{\raggedright\arraybackslash}p{#1}}


\begin{document}

\title[Preservation and definability for the fluted fragment]
{Preservation and definability for the fluted fragment}
\author[Y. Ding]{Yiwen Ding\lmcsorcid{0009-0000-6447-6479}}
\address{School of Business and Economics\\
Vrije Universiteit Amsterdam\\
De Boelelaan 1105\\
1081 HV Amsterdam, The Netherlands}
\email{dyiwen666@gmail.com}

\begin{abstract}
This paper investigates preservation and definability for the equality-free
fluted fragment over finite relational signatures.  We first prove an
equirank homomorphism preservation theorem: every homomorphism-preserved
fluted sentence has an existential-positive fluted equivalent of no greater
quantifier rank.  We then refute Purdy's claimed
\L o\'s--Tarski preservation theorem for the fluted fragment.  Our counterexample
is a rank-three sentence over one binary relation that is preserved under
extensions but has no existential fluted equivalent.  Finally, a two-variable
counterexample over a unary base signature shows that weak and ordinary Beth
definability both fail.  All three results hold over all structures and over
finite structures.
\end{abstract}
\maketitle

\section{Introduction}\label{sec:introduction}

The fluted fragment $\FL$ is a syntactically restricted fragment of
first-order logic in which variables occur in a fixed order and the variables
appearing in an atom form a suffix of the current variable sequence.  For
example,
\[
  \forall x_1\forall x_2
  \bigl(R(x_1,x_2)\to
        \exists x_3(S(x_2,x_3)\wedge P(x_3))\bigr)
\]
is fluted: the binary atom $S(x_2,x_3)$ uses the final two variables, while
the unary atom $P(x_3)$ uses the final one.  The underlying variable discipline
originates in work of Quine~\cite{Quine1969,Quine1975}.  Purdy isolated the
fluted fragment and proved its finite model property and
decidability~\cite{Purdy1996}; he subsequently claimed an exponential-sized model property
and NExpTime-completeness of satisfiability~\cite{Purdy2002}.  However,
Pratt-Hartmann,
Szwast, and Tendera refuted both claims, showing that satisfiability is
non-elementary and obtaining bounds for the finite-variable
subfragments~\cite{PrattHartmannSzwastTendera2019}.  In the standard hierarchy of
trans-elementary complexity classes, these bounds make satisfiability for the
full fragment Tower-complete~\cite{BednarczykKojelisPrattHartmann2025}.  For a
systematic account of this development, we refer to~\cite{PrattHartmann2023}.

The fixed variable order also makes $\FL$ a natural setting for model
comparison.  A closely related syntax appeared as the forward
fragment~\cite{Bednarczyk2021}.  Bednarczyk and Jaakkola showed that, at the level of
sentences, the fluted and forward fragments have the same expressive power;
they also developed bisimulations for these and other ordered logics, obtained
corresponding expressive characterisations, and showed that both fragments
lack Craig interpolation~\cite{BednarczykJaakkola2022}.  Ten Cate and Comer
then proved broad undecidability results for extensions of the fluted fragment
with Craig interpolation~\cite{TenCateComer2025}.  Recent work by Yin further
develops this model-theoretic programme: the ordered fragment has uniform
interpolation and enjoys \L o\'s--Tarski preservation~\cite{YinOrdered2026},
while guarded fluted and forward fragments admit modal
translations and also enjoy \L o\'s--Tarski
preservation~\cite{YinGuarded2026}.  Weak interpolation can be recovered for
the fluted fragment~\cite{YinSavingCraig2026}, while satisfiability for its
uniform subfragment is NExpTime-complete~\cite{YinUniformFluted2026}.

The present paper studies three classical model-theoretic phenomena inside
$\FL$.  We work with the equality-free fluted fragment over finite relational
signatures.  The first is homomorphism preservation.  The classical
homomorphism preservation theorem characterises first-order sentences
preserved under homomorphisms as those equivalent to existential-positive
sentences.  Gr\"adel and Rosen proved the corresponding result for
two-variable logic~\cite{GradelRosen1999}, and Rossman showed that the
homomorphism preservation theorem remains valid when restricted to finite
structures~\cite{Rossman2008}; he subsequently established its equirank form
over finite structures~\cite{Rossman2025}.  Unfortunately, these general
first-order results
do not ensure that the existential-positive equivalent respects the fluted
syntax.

The second phenomenon is extension preservation.  The classical
\L o\'s--Tarski preservation theorem characterises first-order sentences
preserved under extensions as those equivalent to existential
sentences~\cite{Los1955,Tarski1954}.  Purdy claimed that $\FL$ also has
\L o\'s--Tarski preservation~\cite{Purdy2002}.  However, Bednarczyk and
Jaakkola later observed that the proposed argument was too sketchy to verify
and treated the question as open~\cite{BednarczykJaakkola2022}.

The third phenomenon is Beth definability.  Following Andr\'eka and N\'emeti,
an implicit definition is \emph{strong} if every base structure has exactly
one expansion and \emph{weak} if no base structure has more than one.  The weak
Beth property asks for explicit definitions of strong implicit definitions,
whereas the ordinary Beth property covers weak implicit definitions as
well~\cite{AndrekaNemeti2021}.  These properties can be separated in
restricted-variable logics: two-variable logic has weak but not ordinary Beth
definability.  Hoogland and Marx showed that, for guarded
fragments, a suitable weakening of Craig interpolation suffices for Beth
definability~\cite{HooglandMarx2002}.  Motivated by this result, Yin established
a weak interpolation theorem for $\FL$ and left the Beth definability of $\FL$
as an open problem~\cite{YinSavingCraig2026}.

Our contributions in this paper are threefold.  First, we prove that every
homomorphism-preserved fluted sentence of quantifier rank $n$ has an
existential-positive fluted equivalent of quantifier rank at most $n$.
Second, we refute Purdy's claim with a fluted sentence of
quantifier rank three over a single binary relation symbol that is preserved
under extensions but has no existential fluted equivalent.  Third, we give a
two-variable fluted sentence that uniquely determines a binary relation over
the unary base signature $\{P\}$, although no formula in
$\FL[2](\{P\})$ defines it.  All three results persist over finite
structures; in particular, both weak and ordinary Beth definability fail
there as well.

The three proofs expose two complementary effects of the suffix discipline.
Rank-bounded positive characteristic formulas unfold into finite forests, which makes
homomorphism preservation unusually rigid.  In contrast, existential fluted
formulas cannot recover the global row labels used in the \L o\'s--Tarski
counterexample, while unary formulas at level two cannot inspect their first
coordinate.  These limitations are captured respectively by rank-bounded
fluted simulations and rank-bounded fluted bisimulations.

For orientation, \cref{tab:fluted-landscape} gives a selected overview of the
established results most relevant here, together with the contributions of
this paper.  Unless finite structures are mentioned explicitly, the
interpolation and preservation entries use ordinary semantics over all
structures.

\begin{table}[ht]
\centering
\small
\renewcommand{\arraystretch}{1.12}
\begin{tabular}{@{}L{0.24\textwidth}L{0.48\textwidth}L{0.21\textwidth}@{}}
\toprule
\textbf{Topic} & \textbf{Result} & \textbf{Source}\\
\midrule
Satisfiability of $\FL$
  & Finite model property; satisfiability and finite satisfiability coincide
    and are Tower-complete
  & \cite{Purdy1996,PrattHartmannSzwastTendera2019,BednarczykKojelisPrattHartmann2025}\\
$m$-variable subfragments
  & In $(m-2)$-NExpTime for $m\geq3$ and
    $\lfloor m/2\rfloor$-NExpTime-hard for $m\geq2$; exact for
    $2\leq m\leq4$
  & \cite{PrattHartmannSzwastTendera2019}\\
Sentence-level expressivity
  & $\FL$ and the forward fragment have the same expressive power; both admit
    bisimulation characterisations
  & \cite{BednarczykJaakkola2022}\\
Craig interpolation
  & Fails for $\FL$ and for the forward fragment
  & \cite{BednarczykJaakkola2022}\\
Weakened interpolation
  & Yin's signature-relaxed form holds without increasing quantifier rank,
    including at every finite-variable level
  & \cite{YinSavingCraig2026}\\
Ordered fragment
  & Uniform interpolation and the \L o\'s--Tarski theorem hold
  & \cite{YinOrdered2026}\\
Guarded fluted and forward fragments
  & Satisfiability is ExpTime-complete; the \L o\'s--Tarski theorem holds
  & \cite{Bednarczyk2021,YinGuarded2026}\\
Uniform fluted fragment
  & Exponential model property; satisfiability is NExpTime-complete; sentence-level
    expressive equivalence with the uniform forward fragment
  & \cite{YinUniformFluted2026}\\
Homomorphism preservation
  & Holds equirank, over all structures and over finite structures
  & \Cref{thm:hpt-main,thm:finite-hpt}\\
\L o\'s--Tarski preservation
  & Fails over all structures and over finite structures
  & \Cref{thm:los-tarski-failure} and \Cref{cor:finite-los-tarski}\\
Weak and ordinary Beth definability
  & Both fail over all structures and over finite structures
  & \Cref{cor:wbdp-finite}\\
\bottomrule
\end{tabular}
\caption{A selected overview of results for the fluted fragment and closely related restrictions.}
\label{tab:fluted-landscape}
\end{table}

\noindent\textbf{Structure of the paper.}
\Cref{sec:unified-prelim} fixes the syntax, semantic conventions, and
model-comparison notions used throughout.  \Cref{sec:hpt} proves the
homomorphism preservation theorem by constructing rank-bounded type codes,
unfolding their characteristic formulas into canonical forests, and
applying a disjoint-copy reset argument.  \Cref{sec:los-tarski} constructs the
counterexample to \L o\'s--Tarski preservation in $\FL$; extension preservation
follows from a three-valued row analysis, and non-definability follows from
rank-bounded fluted simulations.  \Cref{sec:wbdp} gives a strong implicit definition
and uses bisimulation invariance to refute weak, and therefore ordinary, Beth
definability.

\section{Preliminaries}\label{sec:unified-prelim}

We fix the syntax and semantic conventions for the fluted
fragment~\cite{PrattHartmann2023}.  We then introduce a rank-bounded variant
of the fluted bisimulation of Bednarczyk and
Jaakkola~\cite{BednarczykJaakkola2022}, together with its one-way simulation
analogue.

\subsection{The fluted fragment}\label{sec:fluted-fragment}

Throughout, $\sigma$ is a finite signature consisting of relation symbols of
positive arity, and all $\sigma$-structures have non-empty domains.  Equality,
function symbols, and constants are not allowed.  We use the Roman capital
corresponding to a Fraktur structure for its domain; for example,
$A=\dom(\A)$.  If $\A,\B$ are
structures over the same signature, write $\A\subseteq\B$ when $A\subseteq B$
and $R^{\A}=R^{\B}\cap A^r$ for every $r$-ary relation symbol $R$.  Thus
$\A$ is the induced substructure of $\B$ on $A$, and $\B$ is an
\emph{extension} of $\A$.

Unless explicitly restricted to finite structures, preservation and
equivalence claims range over all $\sigma$-structures, finite or infinite.

Fix the ordered variables $x_1,x_2,\ldots$.  For $m\geq0$, the level
$\FL[m](\sigma)$ is defined by simultaneous induction as follows.
\begin{itemize}
  \item The truth constants $\true,\false$ belong to $\FL[m](\sigma)$.
  If $R\in\sigma$ has arity $r\leq m$, then the suffix atom
  $R(x_{m-r+1},\ldots,x_m)$ belongs to $\FL[m](\sigma)$.
  \item If $\varphi,\psi\in\FL[m](\sigma)$, then
  $\neg\varphi$, $\varphi\wedge\psi$, and $\varphi\vee\psi$ belong to
  $\FL[m](\sigma)$.
  \item If $\varphi\in\FL[m+1](\sigma)$, then
  $\exists x_{m+1}\varphi$ and $\forall x_{m+1}\varphi$ belong to
  $\FL[m](\sigma)$.
\end{itemize}
The fluted fragment over $\sigma$ is
$\FL(\sigma)=\bigcup_{m\geq0}\FL[m](\sigma)$, and $\FL[0](\sigma)$ consists
of sentences.  For some $0\leq\ell\leq m$, the free variables of a formula in
$\FL[m](\sigma)$ form the suffix $x_{m-\ell+1},\ldots,x_m$, understood to be
empty when $\ell=0$.  Thus, $m$ denotes the syntactic level and need not be the
number of free variables.  If
$\bar a=(a_1,\ldots,a_m)\in A^m$ and $\varphi\in\FL[m](\sigma)$, write
$\A\models\varphi[\bar a]$ if $\A,v\models\varphi$, where $v$ is any
assignment satisfying $v(x_i)=a_i$ for $m-\ell+1\leq i\leq m$ and
$\mathsf{free}(\varphi)=\{x_{m-\ell+1},\ldots,x_m\}$.  Thus, $\varphi$ is
evaluated at the suffix $(a_{m-\ell+1},\ldots,a_m)$ of $\bar a$; the entries
$a_1,\ldots,a_{m-\ell}$ are ignored.  The quantifier rank is defined recursively
by $\qr(\true)=\qr(\false)=\qr(\alpha)=0$ for every atom $\alpha$,
$\qr(\neg\varphi)=\qr(\varphi)$,
$\qr(\varphi\wedge\psi)=\qr(\varphi\vee\psi)
=\max\{\qr(\varphi),\qr(\psi)\}$, and
$\qr(Qx_{m+1}\varphi)=\qr(\varphi)+1$ for $Q\in\{\exists,\forall\}$.

The existential-positive levels $\EPFL[m](\sigma)$ are generated from the
permitted atoms and truth constants by conjunction, disjunction, and
existential quantification according to the fluted formation rule.  The
existential levels $\EFL[m](\sigma)$ are
generated in negation normal form from atoms, negated atoms, truth constants,
conjunction, disjunction, and existential quantification according to the
fluted formation rule.  Thus
$\EFL[m](\sigma)$ contains no universal quantifiers.  Put
$\EPFL(\sigma)=\bigcup_{m\geq0}\EPFL[m](\sigma)$ and
$\EFL(\sigma)=\bigcup_{m\geq0}\EFL[m](\sigma)$.  Their sentence levels are
$\EPFL[0](\sigma)$ and $\EFL[0](\sigma)$, respectively.

A map $h:A\to B$ is a \emph{homomorphism} $h:\A\to\B$ if, for every
$r$-ary relation symbol $R$ and all $a_1,\ldots,a_r\in A$,
\[
  R^{\A}(a_1,\ldots,a_r)
  \quad\Longrightarrow\quad
  R^{\B}(h(a_1),\ldots,h(a_r)).
\]
For a non-empty family $(\A_i)_{i\in I}$ of structures, their
\emph{disjoint union}
$\biguplus_{i\in I}\A_i$ has domain
$\{(i,a)\mid i\in I,\ a\in A_i\}$, with relations interpreted componentwise.
Thus no relation tuple contains elements from distinct components.  We freely
identify each component with its canonical image in the disjoint union.

\subsection{Rank-bounded fluted bisimulation and simulation}\label{subsec:model-comparison}

We use a rank-bounded variant of the ordered-logic bisimulation introduced by
Bednarczyk and
Jaakkola~\cite[Definition~2 and Lemma~3]{BednarczykJaakkola2022}, in which related
tuples are extended by appending one element at each move.  We also use the
corresponding simulation notion for existential fluted formulas.

For a tuple $\bar a=(a_1,\ldots,a_m)$ and an element $a'$, write
$\bar a a':=(a_1,\ldots,a_m,a')$.  As usual, $A^0=\{\epsilon\}$, where
$\epsilon$ is the empty tuple.

\begin{defi}[Rank-$k$ fluted bisimulation]\label{def:fluted-bisimulation}
Let $\A,\B$ be $\sigma$-structures and $k\geq0$.  A \emph{rank-$k$ fluted
bisimulation} between $\A$ and $\B$ is a sequence
$\mathcal Z=(Z_0,\ldots,Z_k)$, where
\[
  Z_j\subseteq\bigcup_{m<\omega}(A^m\times B^m)
  \qquad (0\leq j\leq k),
\]
and $Z_k$ is non-empty, satisfying the following conditions:
\begin{itemize}
  \item \emph{Atomic harmony.}  For every $0\leq j\leq k$, every
  $(\bar a,\bar b)\in Z_j$ of length $m$, and every $r$-ary $R\in\sigma$ with
  $r\leq m$,
  \[
    \A\models R(a_{m-r+1},\ldots,a_m)
    \quad\Longleftrightarrow\quad
    \B\models R(b_{m-r+1},\ldots,b_m).
  \]
  \item \emph{Forth condition.}  For every $0\leq j<k$, every
  $(\bar a,\bar b)\in Z_{j+1}$, and every $a'\in A$, there is $b'\in B$ such
  that $(\bar a a',\bar b b')\in Z_j$.
  \item \emph{Back condition.}  For every $0\leq j<k$, every
  $(\bar a,\bar b)\in Z_{j+1}$, and every $b'\in B$, there is $a'\in A$ such
  that $(\bar a a',\bar b b')\in Z_j$.
\end{itemize}
We write $(\A,\bar a)\sim_k^{\FL}(\B,\bar b)$ if some rank-$k$ fluted
bisimulation satisfies $(\bar a,\bar b)\in Z_k$.
\end{defi}

The index $j$ records the remaining quantifier rank: atomic harmony is checked
for every related tuple pair, while each application of the forth or back
condition decreases $j$ by one.  Write
$\A\equiv_k^{\FL}\B$ if $\A,\B$ satisfy the same fluted sentences of
quantifier rank at most $k$.

\begin{prop}[Bisimulation invariance]
\label{prop:bisimulation-invariance}
Let $\A,\B$ be $\sigma$-structures, let $\bar a\in A^m$ and
$\bar b\in B^m$, and let $k\geq0$.  If
$(\A,\bar a)\sim_k^{\FL}(\B,\bar b)$, then
\[
  \A\models\theta[\bar a]
  \quad\Longleftrightarrow\quad
  \B\models\theta[\bar b]
\]
for every $\theta\in\FL[m](\sigma)$ of quantifier rank at most $k$.  In
particular,
$(\A,\epsilon)\sim_k^{\FL}(\B,\epsilon)$ implies
$\A\equiv_k^{\FL}\B$.
\end{prop}

\begin{proof}
Fix a rank-$k$ fluted bisimulation $(Z_0,\ldots,Z_k)$.  We prove, by induction
on $j$, that every $(\bar c,\bar d)\in Z_j$ of common length $p$ agrees on
all formulas in $\FL[p](\sigma)$ of quantifier rank at most $j$.  At each $j$
we use structural induction on the formula.  Atomic formulas follow from
atomic harmony, truth constants have the same value in both structures, and the Boolean cases
follow from the structural induction hypothesis.

Suppose $j>0$.  If $\theta=\exists x_{p+1}\eta$ and
$\A\models\theta[\bar c]$, choose a witness $c'\in A$.  The forth condition gives $d'\in B$
with $(\bar c c',\bar d d')\in Z_{j-1}$.  Since $\qr(\eta)\leq j-1$, the
induction hypothesis for $j-1$ transfers $\eta$, so $d'$ witnesses
$\B\models\theta[\bar d]$.  The converse follows from the back condition.

If $\theta=\forall x_{p+1}\eta$, assume
$\A\models\theta[\bar c]$ and take any $d'\in B$.  The back condition supplies $c'\in A$
with $(\bar c c',\bar d d')\in Z_{j-1}$.  The induction hypothesis for
$j-1$ gives $\B\models\eta[\bar d d']$.  Since $d'$ was arbitrary,
$\B\models\theta[\bar d]$; the converse follows from the forth condition.
\end{proof}

For existential formulas we use the corresponding notion of fluted simulation.

\begin{defi}[Rank-$k$ fluted simulation]\label{def:fluted-simulation}
Let $\A,\B$ be $\sigma$-structures and $k\geq0$.
A \emph{rank-$k$ fluted simulation} from $\A$ to $\B$ is a sequence
$\mathcal Z=(Z_0,\ldots,Z_k)$ satisfying the definition of a rank-$k$ fluted
bisimulation except for the back condition.  We write
$(\A,\bar a)\preccurlyeq_k^{\FL}(\B,\bar b)$ if some such sequence satisfies
$(\bar a,\bar b)\in Z_k$.
\end{defi}

\begin{prop}[Existential transfer]
\label{prop:existential-transfer}
Let $\A,\B$ be $\sigma$-structures, let $\bar a\in A^m$ and
$\bar b\in B^m$, and let $k\geq0$.  If
$(\A,\bar a)\preccurlyeq_k^{\FL}(\B,\bar b)$, then
\[
  \A\models\psi[\bar a]
  \quad\Longrightarrow\quad
  \B\models\psi[\bar b]
\]
for every $\psi\in\EFL[m](\sigma)$ of quantifier rank at most $k$.
Consequently, if
$(\A,\epsilon)\preccurlyeq_k^{\FL}(\B,\epsilon)$ for every $k\geq0$, then
every existential fluted sentence true in $\A$ is true in $\B$.
\end{prop}

\begin{proof}
Fix a witnessing sequence $(Z_0,\ldots,Z_k)$.  As above, we use induction on
$j$, with a structural induction at each $j$, to prove the claim for every
pair $(\bar c,\bar d)\in Z_j$ of common length $p$.  Atomic harmony handles
atoms and negated atoms, truth constants are preserved, and the Boolean
cases follow from the structural induction hypothesis.  The remaining case is
$\psi=\exists x_{p+1}\eta$, which can occur only when $j>0$.  If $\psi$ is
true at $(\A,\bar c)$, choose a witness $c'\in A$.  The forth condition gives $d'\in B$
with $(\bar c c',\bar d d')\in Z_{j-1}$.  Since $\qr(\eta)\leq j-1$, the
induction hypothesis for $j-1$ gives $\B\models\eta[\bar d d']$, and hence
$\B\models\psi[\bar d]$.
\end{proof}
\section{Homomorphism preservation theorem}\label{sec:hpt}
Our proof follows a broad strategy from Rossman's work on homomorphism
preservation~\cite{Rossman2008}, but the technical mechanism is specific to
fluted syntax.  We associate each rank-$n$ root type code with an
existential-positive characteristic sentence and a finite canonical forest
representing its canonical query.  A disjoint-copy reset bisimulation shows
that the forest associated with any model of a homomorphism-preserved sentence
$\varphi$ of rank $n$ also satisfies $\varphi$.  Finiteness of the root type
codes then yields the required finite disjunction.  The construction remains
finite over finite input structures, giving the finite version as well.

\subsection{Rank-bounded type codes and positive characteristic formulas}\label{sec:types}
We encode the rank-bounded information of tuples by finitely many
recursive type codes and translate each type code into an existential-positive
characteristic formula.
\begin{defi}[Atomic type]\label{def:atomic-type}
For $m\geq0$, the atoms available at level $m$ are
\[
  \At_m
  :=
  \bigl\{
     R(x_{m-r+1},\ldots,x_m)\mid
     R\in\sigma,\ \operatorname{ar}(R)=r\leq m
  \bigr\}.
\]
For a $\sigma$-structure $\A$ and a tuple
$\bar a=(a_1,\ldots,a_m)\in A^m$, define the \emph{atomic type} of $\bar a$ in
$\A$,
\[
  \at_{\A}(\bar a)
  :=
  \bigl\{
      \alpha\in\At_m\mid
      \A\models\alpha[\bar a]
  \bigr\}.
\]
\end{defi}

This set records exactly the true atoms at the current level.  Since $\At_m$
is fixed, equality of two such sets records agreement on both true and false
atoms.

\begin{defi}[Rank-bounded type code]\label{def:type-code}
Let $\A$ be a $\sigma$-structure and $\bar a\in A^m$.  For $k\geq0$, define
the \emph{rank-$k$ type code} $\tp_{\A}^k(\bar a)$ of $\bar a$ in $\A$ at
\emph{level $m$} recursively by
\[
  \tp_{\A}^{0}(\bar a)
  :=
  \at_{\A}(\bar a),
\]
and
\begin{equation}
  \tp_{\A}^{k+1}(\bar a)
  :=
  \left(
     \tp_{\A}^{0}(\bar a),
     \bigl\{
       \tp_{\A}^{k}(\bar a b)\mid b\in A
     \bigr\}
  \right).
  \label{eq:type-code}
\end{equation}
The type codes in the second component are called \emph{successor type codes}.
For a type code $t=(p,S)$ of rank at least one, write
$\mathsf{Succ}(t):=S$.
Notice that repeated successor type codes are retained only once: if $b,b'\in A$ satisfy
$\tp_{\A}^{k}(\bar a b)=\tp_{\A}^{k}(\bar a b')$, then this common type code
contributes a single element to the successor set.  Thus the type code records
which successor type codes are realised, but not how many elements realise each
one.
\end{defi}

For fixed $m,k\geq0$, the \emph{rank-$k$ type codes} at \emph{level $m$} are
all values $\tp_{\A}^k(\bar a)$ that arise as $\A$ ranges over the $\sigma$-structures
and $\bar a$ ranges over $A^m$.  Thus only $\sigma$, $m$, and $k$ are fixed;
the structure and tuple may vary.

\begin{lem}[Finiteness of type codes]\label{lem:type-finite}
For every $m,k\geq0$, there are only finitely many rank-$k$ type codes at
level $m$.
\end{lem}

\begin{proof}
For $m,k\geq0$, let $U_{m,k}$ be the set of all rank-$k$ type codes at level
$m$.  We prove by induction on $k$, simultaneously for all $m$, that these
sets are finite.  At rank $0$,
\[
   U_{m,0}\subseteq\mathcal P(\At_m).
\]
Since $\sigma$ is finite, $\At_m$ is finite, and hence so is $U_{m,0}$.

For the induction step, suppose that $U_{m,k}$ is finite for every $m$.  By
the definition of the type codes,
\[
   U_{m,k+1}
   \subseteq
   \mathcal P(\At_m)
   \times
   \mathcal P(U_{m+1,k}).
\]
Here $\At_m$ is finite, and $U_{m+1,k}$ is finite by the simultaneous
induction hypothesis applied at level $m+1$.  Hence the set on the right, and
therefore $U_{m,k+1}$, is finite.  This completes the induction.
\end{proof}

\begin{defi}[Characteristic formula]\label{def:characteristic-formula}
For every $m,k\geq0$ and every rank-$k$ type code $t$ at level $m$, define
its \emph{characteristic formula}
$\chi_t^{m,k}\in\EPFL[m](\sigma)$ by induction on $k$, simultaneously for all
levels $m$.  At rank $0$, we have $t\subseteq\At_m$, and we put
\[
  \chi_t^{m,0}:=\bigwedge_{\alpha\in t}\alpha,
\]
where, throughout this definition, empty conjunctions are interpreted as
$\true$.  For the induction step, let $t=(p,S)$ be a rank-$(k+1)$ type code
at level $m$, and put
\begin{equation}
  \chi_t^{m,k+1}
  :=
  \left(\bigwedge_{\alpha\in p}\alpha\right)
  \wedge
  \bigwedge_{s\in S}
     \exists x_{m+1}\chi_s^{m+1,k}.
  \label{eq:characteristic-formula}
\end{equation}
By \cref{lem:type-finite},~\eqref{eq:characteristic-formula} is a finite
conjunction, so the
recursion is well-defined.

In~\eqref{eq:characteristic-formula}, each conjunct of the form
$\exists x_{m+1}\chi_s^{m+1,k}$ has a separate quantifier.  Different
successor type codes may therefore use different witnesses, even though the
same variable symbol $x_{m+1}$ is reused.  All these conjuncts are evaluated
under the same assignment to $x_1,\ldots,x_m$, although only a suffix of these
variables may actually occur freely.

For a particular structure and tuple, abbreviate
\[
  \chi_{\A,\bar a}^{k}
  :=
  \chi_{\tp_{\A}^{k}(\bar a)}^{m,k},
  \qquad |\bar a|=m.
\]
For every $\sigma$-structure $\A$ and every $n\geq0$, define the
\emph{rank-$n$ positive characteristic sentence} associated with the empty
tuple by
\[
   \Theta_{\A}^{n}
   :=
   \chi_{\A,\epsilon}^{n}
   \in\EPFL[0](\sigma).
\]
Notice that, by construction,
$\chi_{\A,\bar a}^{k}\in\EPFL[m](\sigma)$ and
$\qr(\chi_{\A,\bar a}^{k})\leq k$ whenever $|\bar a|=m$.
\end{defi}

We next verify that the characteristic formulas are realised by their defining
tuples and make explicit the rank-bounded forth and back conditions for equal
type codes.

\begin{lem}\label{lem:chi-basic}
For every $m,k\geq0$, every $\sigma$-structure $\A$, and every tuple
$\bar a\in A^m$:
\begin{enumerate}[label=\textup{(\roman*)}]
  \item $\A\models\chi_{\A,\bar a}^{k}[\bar a]$;
  \item for every $\sigma$-structure $\B$ and every $\bar b\in B^m$, if
  $\tp_{\A}^{k+1}(\bar a)=\tp_{\B}^{k+1}(\bar b)$, then the following
  conditions hold:
  \begin{itemize}
    \item \emph{Atomic harmony.}
    $\at_{\A}(\bar a)=\at_{\B}(\bar b)$.
    \item \emph{Forth condition.}  For every $a'\in A$, there exists $b'\in B$ such
    that
    $\tp_{\A}^{k}(\bar a a')=\tp_{\B}^{k}(\bar b b')$.
    \item \emph{Back condition.}  For every $b'\in B$, there exists $a'\in A$ such
    that
    $\tp_{\A}^{k}(\bar a a')=\tp_{\B}^{k}(\bar b b')$.
  \end{itemize}
\end{enumerate}
In particular, for every $\sigma$-structure $\A$ and every $n\geq0$,
$\Theta_{\A}^{n}\in\EPFL[0](\sigma)$,
$\qr(\Theta_{\A}^{n})\leq n$, and $\A\models\Theta_{\A}^{n}$.
\end{lem}

\begin{proof}
We prove (i) by induction on $k$, simultaneously for all levels $m$.  If
$k=0$, the formula
$\chi_{\A,\bar a}^{0}$ is the conjunction of the current atoms true at
$\bar a$, and hence is true at $\bar a$.

Now assume the assertion at rank $k$ and consider rank $k+1$.  Every current
atomic conjunct in~\eqref{eq:characteristic-formula} is true by the definition of
$\at_{\A}(\bar a)$.  Let $s$ belong to the successor set of
$\tp_{\A}^{k+1}(\bar a)$.  By definition, there is some $b\in A$ such that
$s=\tp_{\A}^{k}(\bar a b)$.  Hence
$\chi_s^{m+1,k}=\chi_{\A,\bar a b}^k$, and the induction hypothesis gives
$\A\models\chi_s^{m+1,k}[\bar a b]$.  Thus $b$ witnesses the corresponding
existential conjunct, and every conjunct in
\eqref{eq:characteristic-formula} is true at $\bar a$.

It remains to verify (ii).  Equality of the two ordered pairs in
\eqref{eq:type-code} gives atomic harmony and equality of their sets of successor type
codes.  Given $a'\in A$, the type code
$\tp_{\A}^{k}(\bar a a')$ belongs to the first successor set and hence to
the second; by the definition of that second set, some $b'\in B$ realises it.
This proves the forth condition, and the symmetric argument proves the back
condition.
The satisfaction claim in the final assertion is the case $m=0$,
$\bar a=\epsilon$ of (i), while membership and the quantifier-rank bound
follow from the observation at the end of
\cref{def:characteristic-formula}.
\end{proof}

\begin{rem}\label{rem:positive-characteristic-information}
The formula $\chi_{\A,\bar a}^{k}$ expresses only the positive information
contained in $\tp_{\A}^{k}(\bar a)$.  If a $\sigma$-structure $\B$ with a
tuple $\bar b$ satisfies it, then $\bar b$ satisfies every current atom
recorded at $\bar a$ and provides a witness for the characteristic formula
associated with each recorded successor type code.  The formula does not assert
that any other atom is false, require those witnesses to have exactly the
recorded successor type codes, or exclude additional successor type codes in $\B$.
Consequently,
$\B\models\chi_{\A,\bar a}^{k}[\bar b]$ need not imply
$\tp_{\A}^{k}(\bar a)=\tp_{\B}^{k}(\bar b)$.  We use equality of type
codes to verify the forth and back conditions, whereas satisfaction of $\chi$
is used to construct homomorphisms.
\end{rem}

\subsection{The canonical forest}\label{sec:canonical}

Fix a $\sigma$-structure $\A$ and an integer $n\geq1$.  We unfold
$\Theta_{\A}^{n}$ into a finite labelled forest $T_{\A}^{n}$ and then
interpret a $\sigma$-structure $\C_{\A}^{n}$ on its nodes.  The labelling map
$\lambda:T_{\A}^{n}\longrightarrow A$ records chosen representatives.  The outermost
existential conjuncts give the roots, and nested existential conjuncts give
their descendants.  Each label $\lambda(v)$ is a chosen representative in
$A$, and distinct nodes may have the same label.

Construct $T_{\A}^{n}$ by depth, with depth recording tuple length.  The empty
tuple is the depth-$0$ starting point.  Write $u\prec v$ when $v$ is an
immediate child of $u$.

The roots have depth $1$.  For each $s$ in the set
\[
  \mathsf{Succ}\bigl(\tp_{\A}^{n}(\epsilon)\bigr)
  =\bigl\{\tp_{\A}^{n-1}(a)\mid a\in A\bigr\},
\]
choose $a_s\in A$ with $\tp_{\A}^{n-1}(a_s)=s$, and create a root $v_s$
with $\lambda(v_s):=a_s$.

Recursively, for a node $v$ at depth $d<n$, write
$v_1\prec\ldots\prec v_d=v$ for its unique root-to-$v$ chain, and put
$c_i:=\lambda(v_i)$ and $t_v:=\tp_{\A}^{n-d}(c_1,\ldots,c_d)$.  Its successor
set is
\[
  \mathsf{Succ}(t_v)
  =
  \bigl\{
    \tp_{\A}^{n-d-1}(c_1,\ldots,c_d,a)
    \mid a\in A
  \bigr\}.
\]
For each $s\in\mathsf{Succ}(t_v)$, choose $a_{v,s}\in A$ with
$\tp_{\A}^{n-d-1}(c_1,\ldots,c_d,a_{v,s})=s$, and attach to $v$ a child
$w_{v,s}$ with $\lambda(w_{v,s}):=a_{v,s}$.  The construction stops at depth
$n$.  Fix all representative choices for the remainder of the construction.

The root set is non-empty because $A$ is non-empty.  By
\cref{lem:type-finite}, the root set and the set of children of every node are
finite.  Since the forest has depth $n$, $T_{\A}^{n}$ is finite and non-empty,
even when $\A$ is infinite.

Conceptually, $T_{\A}^{n}$ is a finite, type-collapsed unravelling of the
existential witnesses required by $\Theta_{\A}^{n}$: paths record successive
witnesses, while at each stage only one node for each distinct successor type
code is retained.

\begin{defi}[Canonical forest structure]\label{def:canonical-structure}
The domain of $\C_{\A}^{n}$ is the node set of $T_{\A}^{n}$.  For every
$r$-ary $R\in\sigma$ and all nodes $v_1,\ldots,v_r$, put
\[
  R^{\C_{\A}^{n}}(v_1,\ldots,v_r)
  \quad\Longleftrightarrow\quad
  v_1\prec\cdots\prec v_r
  \quad\text{and}\quad
  \A\models R(\lambda(v_1),\ldots,\lambda(v_r)).
\]
For $r=1$, the chain condition is vacuous.
\end{defi}

No relation symbol of arity greater than $n$ occurs in $\Theta_{\A}^{n}$,
since an $r$-ary atom can be reached from level $0$ only after introducing at
least $r$ variables.  Nevertheless, to make $\C_{\A}^{n}$ a
$\sigma$-structure, such symbols must still be interpreted.  The preceding
definition makes $R^{\C_{\A}^{n}}$ empty whenever $\operatorname{ar}(R)>n$,
because $T_{\A}^{n}$ has no downward chain containing
$\operatorname{ar}(R)$ nodes.

By definition, the label map
$\lambda:\C_{\A}^{n}\longrightarrow\A$ is a homomorphism.

\begin{lem}[Canonical-query property]\label{lem:canonical-query}
For every $\sigma$-structure $\B$,
\[
   \B\models\Theta_{\A}^{n}
   \quad\Longleftrightarrow\quad
   \text{there is a homomorphism }\C_{\A}^{n}\longrightarrow\B.
\]
In particular, $\C_{\A}^{n}\models\Theta_{\A}^{n}$.

\end{lem}

\begin{proof}
Recall from~\eqref{eq:characteristic-formula} that sibling existential conjuncts may receive
different witnesses even when they reuse the same variable symbol.  For a
depth-$d$ node $v$ with root-to-$v$ chain
$v_1\prec\ldots\prec v_d=v$, put $c_i:=\lambda(v_i)$ and
$t_v:=\tp_{\A}^{n-d}(c_1,\ldots,c_d)$.  By construction, $v$ corresponds to
the occurrence $\exists x_d\chi_{t_v}^{d,n-d}$.  For a root $v_s$, this is
the root conjunct indexed by $s$.  If $d<n$, the children $w_{v,s}$ and the existential conjuncts in
$\chi_{t_v}^{d,n-d}$ are both indexed by $s\in\mathsf{Succ}(t_v)$, and the
choice of $\lambda(w_{v,s})$ gives $t_{w_{v,s}}=s$.

For $r\leq d$ and an $r$-ary $R\in\sigma$, the atom
$R(x_{d-r+1},\ldots,x_d)$ is a conjunct of $\chi_{t_v}^{d,n-d}$ exactly when
$\A\models R(c_{d-r+1},\ldots,c_d)$.  By
\cref{def:canonical-structure}, this is equivalent to the corresponding final
$r$ nodes forming a relation tuple in $\C_{\A}^{n}$.  Conversely, every
relation tuple in $\C_{\A}^{n}$ ends at a unique deepest node and is therefore
covered by this correspondence.

Suppose $\B\models\Theta_{\A}^{n}$.  Choose witnesses recursively from the
roots downward and map each node to its corresponding witness.  The preceding
correspondence shows that this map preserves every relation, and hence is a
homomorphism from $\C_{\A}^{n}$ to $\B$.

Conversely, let $h:\C_{\A}^{n}\to\B$ be a homomorphism, and use $h(v)$ as the
witness at the occurrence corresponding to $v$.  Along each root-to-$v$ chain,
this gives consistent values to the ancestor variables; sibling occurrences
need agree only on their common ancestors.  Since $h$ preserves every
corresponding relation tuple, every atomic conjunct is true.  Induction from
the leaves to the roots of the occurrence forest now gives
$\B\models\Theta_{\A}^{n}$.

Taking $\B=\C_{\A}^{n}$ and using the identity homomorphism gives
$\C_{\A}^{n}\models\Theta_{\A}^{n}$.
\end{proof}

We next record two structural consequences of this construction that
will be used in the reset argument.

\begin{prop}\label{prop:local-exact}
Let $v_1\prec\cdots\prec v_d$ be a consecutive downward parent--child segment in $T_{\A}^{n}$, with labels
$c_i=\lambda(v_i)$.  For every $r$-ary $R\in\sigma$ with $r\leq d$,
\[
 \C_{\A}^{n}\models R(v_{d-r+1},\ldots,v_d)
 \quad\Longleftrightarrow\quad
 \A\models R(c_{d-r+1},\ldots,c_d).
\]
\end{prop}

\begin{proof}
The displayed nodes form a consecutive downward chain.  The equivalence is therefore exactly the defining clause for $R^{\C_{\A}^{n}}$.
\end{proof}

Here and below, every copy of $\C_{\A}^{n}$ carries the parent--child relation
and labelling map transported from $T_{\A}^{n}$.

\begin{prop}\label{prop:break}
Let $\D$ be a disjoint union of copies of $\C_{\A}^{n}$.  If
$\D\models R(v_1,\ldots,v_r)$ for an $r$-ary $R\in\sigma$, then all $v_i$
lie in the same copy and $v_i\prec v_{i+1}$ for every $i<r$.
\end{prop}

\begin{proof}
Every relation tuple in $\D$ comes from a single copy of $\C_{\A}^{n}$, and
the conclusion then follows from \cref{def:canonical-structure}.
\end{proof}

\subsection{The disjoint-copy reset lemma}\label{sec:reset}

The next step is the central technical comparison.  It shows that fluted
sentences of quantifier rank at most $n$ cannot distinguish $\A$ together with
sufficiently many canonical copies from the copies alone.  Fix a
$\sigma$-structure $\A$ and $n\geq1$, and write $\C:=\C_{\A}^{n}$.  Let
$\D$ be a disjoint union of $n+1$ isomorphic copies of $\C$, chosen disjoint
from $\A$, and put
\[
   \A^{\dagger}:=\A\uplus\D.
\]
The $n+1$ copies ensure that an unused component remains available whenever
the rank-$n$ comparison resets the active suffix.

\begin{lem}[Disjoint-copy reset lemma]\label{lem:reset}
There is a rank-$n$ fluted bisimulation $(Z_0,\ldots,Z_n)$ between
$\A^{\dagger}$ and $\D$ with $(\epsilon,\epsilon)\in Z_n$.  Consequently,
\[
   \A^{\dagger}\equiv_n^{\FL}\D.
\]
\end{lem}

\begin{proof}
We define $Z_{n-m}$ by certificates for tuple pairs of length $m$.  We need
to check three points: $(\epsilon,\epsilon)\in Z_n$; every one-element move
from a pair in $Z_{j+1}$ has a certificate-preserving match in $Z_j$; and
every certified pair satisfies atomic harmony.

\paragraph{Candidate relations.}
For $0\leq m\leq n$, let $Z_{n-m}$ consist exactly of those pairs
$(\bar u,\bar w)\in(A^{\dagger})^m\times D^m$ that are certified in the
following sense.  The empty pair $(\epsilon,\epsilon)$ is certified by
convention, and hence belongs to $Z_n$.  Atomic harmony at length $0$ holds
because the signature has no nullary relation symbols.

For $m\geq1$, a \emph{certificate} for $(\bar u,\bar w)$ consists of a choice
of $1\leq\ell\leq m$.  The final $\ell$ entries of both tuples, called their
\emph{active suffixes}, must have one of the two forms below.  If $\ell<m$,
call the adjacent positions $m-\ell$ and $m-\ell+1$ the \emph{cut}; the pair
in $\A^{\dagger}$ at the cut must cross components, whereas the pair in $\D$
must either cross copies of $\C$ or fail to be a downward parent--child pair.
No condition is imposed on the entries before the cut: every current atom
involving an earlier position also contains both positions of the cut.
\begin{itemize}
  \item \emph{Copy configuration.}  They lie in one copy of $\C$ in each
  structure, and a fixed copy isomorphism maps the suffix in $\A^{\dagger}$
  pointwise to the suffix in $\D$.
  \item \emph{$A$--branch configuration.}  The suffix in $\A^{\dagger}$ is
  $(a_1,\ldots,a_\ell)\in A^\ell$, while the suffix in $\D$ is a consecutive
  branch $v_1\prec\cdots\prec v_\ell$ beginning at a root.  Writing
  $c_i:=\lambda(v_i)$, we require
  \begin{equation}
     \tp_{\A}^{n-\ell}(a_1,\ldots,a_\ell)
     =
     \tp_{\A}^{n-\ell}(c_1,\ldots,c_\ell).
     \label{eq:reset-invariant}
  \end{equation}
\end{itemize}

\paragraph{Forth and back.}
Fix a certified pair in $Z_{n-m}$ with $m<n$, together with one of its
certificates.  We must match any next element
chosen in either structure so that the extended pair has a certificate of
length $m+1$.  On the side where the match is made, the current tuple meets at
most $m<n$ copies of $\C$, so an unused copy is available.  Since equality is
absent, the response need not repeat an earlier element.

At length $0$, a move starts one of the two configurations.  Thereafter, in a
copy configuration, a move stays in the active copy, enters $A$ (only in
$\A^{\dagger}$), or enters another copy; in an $A$--branch configuration, it
either continues the branch or breaks it.  These possibilities give the
following three exhaustive cases.

\begin{enumerate}[label=\textup{(\roman*)}]
\item \emph{Continue a copy.}
In the copy configuration, if the chosen node lies in the active copy, use
its image or preimage under the fixed copy isomorphism.  Appending these nodes
extends the active suffix and preserves the certificate.

\item \emph{Start or continue an $A$--branch.}
Suppose $a\in A$ is chosen in $\A^{\dagger}$ when the current pair is empty
or has a copy certificate.  The set
$\mathsf{Succ}(\tp_{\A}^{n}(\epsilon))$ contains
$\tp_{\A}^{n-1}(a)$, so a copy not met by the current tuple in $\D$ has a
root $v$ whose label $c$ satisfies
\[
   \tp_{\A}^{n-1}(a)=\tp_{\A}^{n-1}(c).
\]
Taking $(a,v)$ as the new active suffixes yields an $A$--branch certificate.
If $m>0$, the new cut separates $A$ from a copy of $\C$ in
$\A^{\dagger}$ and two distinct copies in $\D$.

Now suppose an $A$--branch of length $\ell$ is active.  Since
$\ell\leq m<n$, put $q:=n-\ell-1\geq0$.  If $a_{\ell+1}\in A$ is chosen,
apply the forth condition in \cref{lem:chi-basic}(ii) to
\eqref{eq:reset-invariant} with
$k=q$.  It shows that the rank-$q$ type code of
$(a_1,\ldots,a_\ell,a_{\ell+1})$ is realised as a successor type code over
$(c_1,\ldots,c_\ell)$.  The forest construction therefore supplies a child
$v_{\ell+1}$ whose label $c_{\ell+1}:=\lambda(v_{\ell+1})$ satisfies
\begin{equation}
  \tp_{\A}^{q}(a_1,\ldots,a_\ell,a_{\ell+1})
  =
  \tp_{\A}^{q}(c_1,\ldots,c_\ell,c_{\ell+1}).
  \label{eq:reset-invariant-plus}
\end{equation}
Conversely, a child $v_{\ell+1}$ chosen in $\D$ represents a successor type
code, so the back condition in \cref{lem:chi-basic}(ii), again applied to
\eqref{eq:reset-invariant} with $k=q$, supplies $a_{\ell+1}\in A$ satisfying the same
equality.  Thus either move extends the $A$--branch certificate.

\item \emph{Reset to a copy.}
All remaining moves choose a node of a copy of $\C$ without continuing the
current configuration: a node outside the active copy in the copy
configuration, a node in a copy of $\C$ inside $\A^{\dagger}$ during an
$A$--branch, or a non-child of $v_\ell$ in $\D$ during an $A$--branch.  Match
the chosen node by taking, on the opposite side, the corresponding node in a
copy not met by the current tuple there, and make this singleton pair the new
active suffix.  This also handles the first move when the chosen element is a
node of $\C$.

If $m>0$, the reset creates a cut.  On the side where the node was chosen,
the cut either crosses components or, in the last case above, is not a
parent--child pair; on the matching side it crosses into an unused copy.
Hence the cut conditions hold.
\end{enumerate}
In each case the extended pair belongs to $Z_{n-m-1}$.  This proves both the
forth and back conditions.

\paragraph{Atomic harmony.}
Fix a certified pair of length $1\leq m\leq n$, let $\ell$ be the length of
its active suffix, and consider an $r$-ary current fluted atom.

If $r\leq\ell$, the atom lies wholly inside the active suffix.  In the copy
configuration, the copy isomorphism preserves its truth.  In the
$A$--branch configuration,~\eqref{eq:reset-invariant} entails
\[
   \at_{\A}(a_1,\ldots,a_\ell)
   =
   \at_{\A}(c_1,\ldots,c_\ell),
\]
also when $n-\ell=0$; \cref{prop:local-exact} then transfers the atom from the
label tuple to the forest branch in $\D$.

If $r>\ell$, then $\ell<m$ and
$m-r+1\leq m-\ell<m-\ell+1\leq m$, so the atom contains the cut pair.  It is
false in $\A^{\dagger}$ because the cut crosses components, and false in
$\D$ by \cref{prop:break}.

\paragraph{Conclusion.}
We have shown that $(\epsilon,\epsilon)\in Z_n$, that every move from
$Z_{j+1}$ can be matched into $Z_j$, and that every pair in every $Z_j$
satisfies atomic harmony.  Hence $(Z_0,\ldots,Z_n)$ is a rank-$n$ fluted
bisimulation, and
\cref{prop:bisimulation-invariance} yields the claimed equivalence.
\end{proof}

\subsection{Proof of the homomorphism preservation theorem}\label{sec:hpt-proof}

We can now assemble the construction into the equirank theorem.

\begin{thm}[Homomorphism preservation theorem]\label{thm:hpt-main}
Let $\varphi$ be a sentence in $\FL[0](\sigma)$, and put
$n:=\qr(\varphi)$.  The following are
equivalent, with preservation and equivalence taken over all
$\sigma$-structures:
\begin{enumerate}[label=\textup{(\roman*)}]
  \item $\varphi$ is preserved under homomorphisms;
  \item $\varphi$ is equivalent to some $\psi\in\EPFL[0](\sigma)$ with
  $\qr(\psi)\leq n$.
\end{enumerate}
\end{thm}
\begin{lem}\label{lem:ep-hom}
Let $h:\A\to\B$ be a homomorphism of $\sigma$-structures.  For every
$\eta\in\EPFL[m](\sigma)$ and every $\bar a\in A^m$, write
$h(\bar a)=(h(a_1),\ldots,h(a_m))$.  Then
\[
  \A\models\eta[\bar a]
  \quad\Longrightarrow\quad
  \B\models\eta[h(\bar a)].
\]
In particular, every sentence in $\EPFL[0](\sigma)$ is preserved under
homomorphisms.
\end{lem}

\begin{proof}
We prove by structural induction on $\eta$.  Atomic formulas are preserved by
the definition of homomorphism, and truth constants are preserved as well.
The conjunction and disjunction cases follow directly from the induction
hypothesis.  Finally, suppose
$\A\models(\exists x_{m+1}\theta)[\bar a]$ and let $b\in A$ be a witness.
The induction hypothesis gives
\[
  \B\models\theta[h(\bar a)h(b)].
\]
Therefore $h(b)$ witnesses $\exists x_{m+1}\theta$ in $\B$.  This covers every
constructor of $\EPFL$.
\end{proof}

By \cref{lem:ep-hom}, this proves \textup{(ii)}$\Rightarrow$\textup{(i)} in
\cref{thm:hpt-main}.  Indeed, if $\varphi$ is equivalent to
$\psi\in\EPFL[0](\sigma)$, $h:\A\to\B$, and $\A\models\varphi$, then
$\A\models\psi$, hence $\B\models\psi$, and therefore $\B\models\varphi$.

For the reverse implication, the only substantive case is where $n\geq1$ and
$\varphi$ is satisfiable.  We first prove that $\Theta_{\A}^{n}$ entails
$\varphi$ whenever $\A\models\varphi$; finiteness of the root type codes will
then yield the required finite disjunction.

\begin{lem}[Canonical forests remain in the class]\label{lem:canonical-in-class}
Let $\varphi\in\FL[0](\sigma)$, where $\qr(\varphi)=n\geq1$, and suppose
that $\varphi$ is preserved under homomorphisms.  If $\A$ is a
$\sigma$-structure satisfying $\varphi$, then
\[
   \C_{\A}^{n}\models\varphi.
\]
\end{lem}

\begin{proof}
Put $\C:=\C_{\A}^{n}$.  Let $\D$ be a disjoint union of $n+1$ copies of
$\C$, chosen disjoint from $\A$, and put $\A^{\dagger}:=\A\uplus\D$.
The argument follows the transfer chain
\[
   \A\longrightarrow\A^{\dagger}
   \equiv_n^{\FL}\D
   \longrightarrow\C.
\]
The first map is the canonical injection into the $\A$-component of
$\A^{\dagger}$, so homomorphism preservation gives
$\A^{\dagger}\models\varphi$.  The middle equivalence is
\cref{lem:reset}; since $\qr(\varphi)=n$, it gives $\D\models\varphi$.
For the final map, fold each component of $\D$ isomorphically onto $\C$.
Every relation tuple of $\D$ lies within one component, so the fold preserves
relations and is a homomorphism.  Hence $\C\models\varphi$.
\end{proof}

\begin{cor}[Characteristic sentences of $\varphi$-structures entail $\varphi$]\label{cor:theta-entails}
Let $\varphi\in\FL[0](\sigma)$ have quantifier rank $n\geq1$ and be preserved
under homomorphisms over all $\sigma$-structures.  For every
$\sigma$-structure $\A$ satisfying $\varphi$,
\[
   \models\Theta_{\A}^{n}\rightarrow\varphi.
\]
\end{cor}

\begin{proof}
If $\B\models\Theta_{\A}^{n}$, then \cref{lem:canonical-query} gives a
homomorphism $\C_{\A}^{n}\to\B$.  By \cref{lem:canonical-in-class},
$\C_{\A}^{n}\models\varphi$, so homomorphism preservation yields
$\B\models\varphi$.
\end{proof}

\begin{proof}[Proof of \textup{(i)}$\Rightarrow$\textup{(ii)} in
\cref{thm:hpt-main}]
Suppose that $\varphi$ is preserved under homomorphisms, and put
$n=\qr(\varphi)$.  If $n=0$, then, because the signature has no nullary
relation symbols, $\varphi$ is equivalent to $\true$ or $\false$, both of
which belong to $\EPFL[0](\sigma)$.  We may therefore assume that $n\geq1$.
If $\varphi$ is unsatisfiable, then it is equivalent to $\false$, so assume
also that $\varphi$ is satisfiable.

It remains to collect the finitely many characteristic sentences arising from
$\sigma$-structures satisfying $\varphi$.  Consider the set of their rank-$n$
root type codes:

\[
   \mathcal T_{\varphi}
   :=
   \bigl\{
      \tp_{\A}^{n}(\epsilon)
      \mid
      \A\text{ is a $\sigma$-structure satisfying }\varphi
   \bigr\}.
\]
By \cref{lem:type-finite} and the satisfiability of $\varphi$,
$\mathcal T_{\varphi}$ is finite and non-empty.  Define
\begin{equation}
   \psi_{\varphi}
   :=
   \bigvee_{t\in\mathcal T_{\varphi}}
      \chi_t^{0,n}
   .
   \label{eq:forest-normal-form}
\end{equation}
By \cref{def:characteristic-formula}, each disjunct belongs to
$\EPFL[0](\sigma)$ and has quantifier rank at most $n$.  Hence the finite
disjunction satisfies
\[
   \psi_{\varphi}\in\EPFL[0](\sigma)
   \qquad\text{and}\qquad
   \qr(\psi_{\varphi})\leq n.
\]

If $\B\models\varphi$, then
$\tp_{\B}^{n}(\epsilon)\in\mathcal T_{\varphi}$, and
\cref{lem:chi-basic}(i) gives
$\B\models\chi_{\tp_{\B}^{n}(\epsilon)}^{0,n}$.  Hence
$\B\models\psi_{\varphi}$.

Conversely, if $\B\models\psi_{\varphi}$, then
$\B\models\chi_t^{0,n}$ for some $t\in\mathcal T_{\varphi}$.  Choose
$\A\models\varphi$ with $t=\tp_{\A}^{n}(\epsilon)$.  Since
$\chi_t^{0,n}=\Theta_{\A}^{n}$, \cref{cor:theta-entails} gives
$\B\models\varphi$.  We have proved both entailments, and hence
\[
   \models\varphi\leftrightarrow\psi_{\varphi}.\qedhere
\]
\end{proof}

\begin{cor}[Forest-query normal form]\label{cor:forest-normal-form}
Let $\varphi\in\FL[0](\sigma)$ have quantifier rank $n\geq1$ and be
preserved under homomorphisms over all $\sigma$-structures.  Then $\varphi$
is either unsatisfiable or
equivalent to a finite non-empty disjunction
\[
   \Theta_{\A_1}^{n}\vee\cdots\vee\Theta_{\A_s}^{n},
\]
where $\A_i\models\varphi$ for every $1\leq i\leq s$.  Each disjunct is a
primitive-positive fluted sentence: it is built from atoms and truth constants
using conjunction and existential quantification.  Its canonical structure is
carried by a finite forest of height at most $n$, and every relation tuple lies
on a consecutive downward chain.
\end{cor}

\begin{proof}
The unsatisfiable case is one of the stated alternatives.  Otherwise,
let $\mathcal T_{\varphi}$ and $\psi_{\varphi}$ be as in the proof of
\cref{thm:hpt-main}.  Enumerate
$\mathcal T_{\varphi}=\{t_1,\ldots,t_s\}$ and choose
$\A_i\models\varphi$ with $t_i=\tp_{\A_i}^{n}(\epsilon)$ for every $i$.  The
normal form~\eqref{eq:forest-normal-form} then gives
\[
  \psi_{\varphi}
  =\Theta_{\A_1}^{n}\vee\cdots\vee\Theta_{\A_s}^{n},
\]
which is equivalent to $\varphi$.  By \cref{def:characteristic-formula}, each
disjunct has the stated primitive-positive form.  For every $i$, the
construction of $\C_{\A_i}^{n}$ in \cref{sec:canonical} uses a finite forest
of height at most $n$, and \cref{def:canonical-structure} ensures that every
relation tuple lies on a consecutive downward chain.
\end{proof}

For quantifier rank $0$, the corresponding normal form is simply $\true$ or
$\false$; no non-empty canonical forest is required.

\subsection{The finite-structure version}\label{sec:finite-hpt}

We now prove the equirank result over finite structures.

\begin{thm}[Finite homomorphism preservation]\label{thm:finite-hpt}
Let $\varphi$ be a sentence in $\FL[0](\sigma)$, and put
$n:=\qr(\varphi)$.  The following are
equivalent, with preservation and equivalence taken over finite
$\sigma$-structures:
\begin{enumerate}[label=\textup{(\roman*)}]
  \item $\varphi$ is preserved under homomorphisms;
  \item $\varphi$ is equivalent to some $\psi\in\EPFL[0](\sigma)$ with
  $\qr(\psi)\leq n$.
\end{enumerate}
\end{thm}

\begin{proof}
For \textup{(ii)}$\Rightarrow$\textup{(i)}, let $\psi$ witness
\textup{(ii)} and let $h:\A\to\B$ be a homomorphism between finite structures.
If $\A\models\varphi$, then $\A\models\psi$, hence $\B\models\psi$, and
therefore $\B\models\varphi$ by \cref{lem:ep-hom}.

For \textup{(i)}$\Rightarrow$\textup{(ii)}, suppose that $\varphi$ is
preserved under homomorphisms between finite $\sigma$-structures.  If $n=0$,
then $\varphi$ is equivalent to $\true$ or $\false$.  If $\varphi$ is
unsatisfiable over finite structures, it is equivalent to $\false$ there.  We
may therefore assume that $n\geq1$ and that $\varphi$ is satisfiable over
finite structures.

As in the proof over all structures, it remains to establish characteristic
entailment and then collect the finitely many root type codes.  We also need
to check that every structure in the transfer argument remains finite.

Fix a finite $\A\models\varphi$, put $\C:=\C_{\A}^{n}$, let $\D$ be the
disjoint union of $n+1$ copies of $\C$, and put
$\A^{\dagger}:=\A\uplus\D$.  Every structure in the transfer chain
\[
   \A\longrightarrow\A^{\dagger}
   \equiv_n^{\FL}\D
   \longrightarrow\C
\]
is finite: $\C$ is finite by construction, $\D$ has finitely many copies, and
$\A^{\dagger}=\A\uplus\D$.  The first and last arrows are the canonical
injection and fold homomorphisms, while the middle equivalence is
\cref{lem:reset}.  Since $\qr(\varphi)=n$, the chain gives
$\C\models\varphi$ using preservation only between finite structures.

If a finite $\B$ satisfies $\Theta_{\A}^{n}$, then
\cref{lem:canonical-query} gives a homomorphism $\C\to\B$.  Both structures
are finite, so homomorphism preservation yields $\B\models\varphi$.  Thus
$\Theta_{\A}^{n}$ entails $\varphi$ over finite structures.

Now let
\[
  \mathcal T_{\varphi}^{\mathrm{fin}}
  :=
  \bigl\{
    \tp_{\A}^{n}(\epsilon)
    \mid
    \A\text{ is a finite $\sigma$-structure satisfying }\varphi
  \bigr\}.
\]
By \cref{lem:type-finite} and finite satisfiability of $\varphi$,
$\mathcal T_{\varphi}^{\mathrm{fin}}$ is finite and non-empty.  Define
\[
  \psi_{\varphi}^{\mathrm{fin}}
  :=
  \bigvee_{t\in\mathcal T_{\varphi}^{\mathrm{fin}}}\chi_t^{0,n}.
\]
By \cref{def:characteristic-formula}, each disjunct belongs to
$\EPFL[0](\sigma)$ and has quantifier rank at most $n$.  Hence
$\psi_{\varphi}^{\mathrm{fin}}\in\EPFL[0](\sigma)$ and
$\qr(\psi_{\varphi}^{\mathrm{fin}})\leq n$.

If a finite $\B$ satisfies $\varphi$, then
$\tp_{\B}^{n}(\epsilon)\in\mathcal T_{\varphi}^{\mathrm{fin}}$, and
\cref{lem:chi-basic}(i) gives
$\B\models\chi_{\tp_{\B}^{n}(\epsilon)}^{0,n}$.  Hence
$\B\models\psi_{\varphi}^{\mathrm{fin}}$.

Conversely, if a finite $\B$ satisfies $\psi_{\varphi}^{\mathrm{fin}}$, then
$\B\models\chi_t^{0,n}$ for some
$t\in\mathcal T_{\varphi}^{\mathrm{fin}}$.  Choose a finite
$\A\models\varphi$ with $t=\tp_{\A}^{n}(\epsilon)$.  Since
$\chi_t^{0,n}=\Theta_{\A}^{n}$, the characteristic entailment proved above
gives $\B\models\varphi$.  Thus $\varphi$ and
$\psi_{\varphi}^{\mathrm{fin}}$ are equivalent over finite structures.
\end{proof}

\section{Failure of the \L o\'s--Tarski theorem}\label{sec:los-tarski}

In this section, we construct a rank-three fluted sentence $\Phi$ over one
binary relation that is preserved under extensions but has no existential
fluted equivalent, either over all structures or over finite structures.  We
first define $\Phi$ using row labels and profiles.  A structural
characterisation of $\neg\Phi$ then yields extension preservation, while
simulations of every rank between two finite structures rule out existential
fluted definability.

\subsection{A fluted sentence \texorpdfstring{\(\Phi\)}{Phi}}
\label{sec:counterexample-formula}

Throughout this section, let \(\sigma=\{R\}\), where \(R\) is binary.  The
counterexample below uses only the variables
\(x_1,x_2,x_3\), in this fluted order.  For
\(m\in\{2,3\}\), the only binary atom in \(\FL[m](\sigma)\) is
\(R(x_{m-1},x_m)\); hence the occurrences below are \(R(x_1,x_2)\) in
\(\FL[2](\sigma)\)-formulas and \(R(x_2,x_3)\) in
\(\FL[3](\sigma)\)-formulas.  For \(\varepsilon\in\{0,1\}\),
\(R^\varepsilon(x_{m-1},x_m)\) abbreviates \(R(x_{m-1},x_m)\) if
\(\varepsilon=1\), and \(\neg R(x_{m-1},x_m)\) if \(\varepsilon=0\).

For a \(\sigma\)-structure \(\C\) and \(c\in C\), the
\emph{\(R\)-row of \(c\) in \(\C\)} is the function
\(\rho_c^\C:C\to\{0,1\}\), where \(\rho_c^\C(d)=1\) iff
\(R^\C(c,d)\).  The row is \emph{zero}, \emph{one}, or \emph{mixed}
according as \(\rho_c^\C\) is constantly \(0\), constantly \(1\), or takes
both values.

The following \(\FL[2](\sigma)\)-formulas express the three exhaustive and mutually exclusive \(R\)-row behaviours:
\begin{align*}
  \Zt(x_2)&:=\forall x_3\,\neg R(x_2,x_3),\\
  \Ot(x_2)&:=\forall x_3\,R(x_2,x_3),\\
  \Mt(x_2)&:=
    \exists x_3\,R(x_2,x_3)
    \wedge
   \exists x_3\,\neg R(x_2,x_3).
\end{align*}
We also use the corresponding \(\FL[1](\sigma)\)-formula
\[
  \Mt(x_1):=\exists x_2\,R(x_1,x_2)
    \wedge
  \exists x_2\,\neg R(x_1,x_2).
\]
The symbols \(\Zt,\Ot,\Mt\) without displayed variables are labels.  For a
label \(T\), \(T(x_2)\) denotes the corresponding
\(\FL[2](\sigma)\)-formula among
\(\Zt(x_2),\Ot(x_2),\Mt(x_2)\).  The only label also used with variable
\(x_1\) is \(\Mt\), via the \(\FL[1](\sigma)\)-formula \(\Mt(x_1)\)
above.  Since the formulas \(T(x_2)\) have only \(x_2\) free, we write
\(\C\models T(x_2)[c]\) for satisfaction under the one-coordinate assignment
\(x_2\mapsto c\); equivalently, this abbreviates
\(\C\models T(x_2)[e,c]\) for any \(e\in C\).  We say that \(c\)
\emph{has label \(T\) in \(\C\)} when this holds.  Thus an element with label
\(\Zt,\Ot,\Mt\) has, respectively, a zero, one, or mixed \(R\)-row.  Since
structures are non-empty, every element has exactly one of these three labels.
We write
\[
  \lambda_{\C}:C\longrightarrow\Types
\]
for this three-valued label map.

For \(c\in C\), define its \emph{row profile} by
\begin{equation}
  \begin{array}{rcl}
  \prof_{\C}(c)
  &:=&\{(\varepsilon,T)\in\{0,1\}\times\Types\mid
  \exists d\in C\text{ of label }T\\
  &&\text{such that }\rho_c^\C(d)=\varepsilon\}.
  \end{array}
  \label{eq:profile}
\end{equation}

For \(\varepsilon\in\{0,1\}\) and a label \(T\), let
\begin{equation}
  E_{\varepsilon,T}(x_1)
  :=
  \exists x_2\bigl(R^\varepsilon(x_1,x_2)\wedge T(x_2)\bigr).
  \label{eq:E}
\end{equation}
By construction, for every \(c\in C\),
\[
  (\varepsilon,T)\in\prof_{\C}(c)
  \quad\Longleftrightarrow\quad
  \C\models E_{\varepsilon,T}(x_1)[c].
\]

Define
\begin{equation}
  \mathsf{Split}
  :=
  \bigvee_{T\in\Types}
  \exists x_1
  \bigl(E_{0,T}(x_1)\wedge E_{1,T}(x_1)\bigr),
  \label{eq:split}
\end{equation}
and
\begin{equation}
  \mathsf{Diverse}
  :=
  \bigvee_{\substack{\varepsilon\in\{0,1\}\\T\in\Types}}
  \bigl[\exists x_1(\Mt(x_1)\wedge E_{\varepsilon,T}(x_1))\wedge\exists x_1(\Mt(x_1)\wedge\neg E_{\varepsilon,T}(x_1))\bigr].
  \label{eq:diverse}
\end{equation}
Finally, define
\begin{equation}
  \Phi:=\mathsf{Split}\vee\mathsf{Diverse}.
  \label{eq:Phi}
\end{equation}
All disjunctions in~\eqref{eq:split} and~\eqref{eq:diverse} are finite, and every displayed abbreviation can be eliminated.

\begin{prop}
\label{prop:phi-syntax}
The sentence \(\Phi\) belongs to \(\FL[0](\sigma)\) and has quantifier rank
three.
\end{prop}

\begin{proof}
We verify the levels and ranks from the inside out.  The three label formulas
belong to \(\FL[2](\sigma)\), have only \(x_2\) free, and have quantifier rank one.
Thus \(R^\varepsilon(x_1,x_2)\wedge T(x_2)\) belongs to \(\FL[2](\sigma)\) and
has rank one, so \(E_{\varepsilon,T}\) belongs to \(\FL[1](\sigma)\), has only
\(x_1\) free, and has rank two.  The formula \(\Mt(x_1)\) belongs to
\(\FL[1](\sigma)\) and has rank one, while \(\neg E_{\varepsilon,T}\) belongs to
\(\FL[1](\sigma)\) and has rank two.  Hence every matrix quantified by \(x_1\) in
\eqref{eq:split} and~\eqref{eq:diverse} belongs to \(\FL[1](\sigma)\) and has
rank two.  The outer quantifiers yield fluted sentences of rank three, and
their finite Boolean combination \(\Phi\) has rank three.
\end{proof}

The definitions give \(\C\models\mathsf{Split}\) iff
\(\{(0,T),(1,T)\}\subseteq\prof_{\C}(c)\) for some \(c\in C\) and
\(T\in\Types\).  Likewise, \(\C\models\mathsf{Diverse}\) iff two elements of
label \(\Mt\) have different profiles.  Hence \(\neg\Phi\) says exactly that
no row splits a label class and that all elements of label \(\Mt\) have the
same profile.

The purpose of \(\mathsf{Diverse}\) is to make \(\neg\Phi\) stable under
induced substructures.  Under restriction, elements of label \(\Mt\) may
acquire label \(\Zt\) or \(\Ot\); if mixed rows have different profiles, this
relabelling can create a new split.  The failure of \(\mathsf{Diverse}\)
prevents this by forcing all mixed rows to have the same profile.

The second disjunct is necessary: \(\mathsf{Split}\) by itself is not preserved
under extensions.  Consider the structure with \(R\)-matrix
\[
\begin{pmatrix}
0&0&0&0\\
0&0&0&1\\
0&1&1&0\\
1&1&1&1
\end{pmatrix}.
\]
Its four elements have labels \(\Zt,\Mt,\Mt,\Ot\), respectively, and
\(\mathsf{Split}\) is false.  In the induced substructure on the first three
elements, the labels become \(\Zt,\Zt,\Mt\), and the third row takes both values
on the new \(\Zt\)-class.  Thus \(\mathsf{Split}\) becomes true.  The two
original \(\Mt\)-rows have different profiles, so the full structure satisfies
\(\mathsf{Diverse}\).  The example therefore violates preservation for
\(\mathsf{Split}\), but not for \(\Phi\).

\subsection{Preservation under extensions}
\label{sec:extension-preservation}

We prove extension preservation by showing that the class defined by
\(\neg\Phi\) is closed under induced substructures.  The following
characterisation provides the required structural form.

\begin{lem}[Label-matrix characterisation]
\label{lem:label-matrix}
For a \(\sigma\)-structure \(\C\), put
\[
  \mathcal L_{\C}:=
  \{\lambda_{\C}(c)\mid c\in C\}.
\]
The following are equivalent.
\begin{enumerate}[label=\textup{(\roman*)}]
  \item \(\C\models\neg\Phi\).
  \item There is a function
  \[
    \mu_{\C}:\mathcal L_{\C}\times\mathcal L_{\C}
    \longrightarrow\{0,1\}
  \]
  such that, for all \(c,d\in C\),
  \begin{equation}
    R^{\C}(c,d)
    \quad\Longleftrightarrow\quad
    \mu_{\C}\bigl(\lambda_{\C}(c),\lambda_{\C}(d)\bigr)=1.
    \label{eq:label-matrix}
  \end{equation}
\end{enumerate}
\end{lem}

\begin{proof}
Assume \textup{(i)}.  We first show that each row is constant on each
realised target-label class, and then that its values depend only on the
source label.  For \(c\in C\) and \(T\in\mathcal L_{\C}\), the failure
of \(\mathsf{Split}\) gives a unique value \(v_c(T)\in\{0,1\}\) such that
\[
  \rho_c^{\C}(d)=v_c(T)
  \qquad\text{whenever }\lambda_{\C}(d)=T.
\]
Here existence uses the fact that \(T\) is realised.  If \(c\) has label
\(\Zt\), then \(v_c\) is constantly \(0\), and if it has label \(\Ot\), then
\(v_c\) is constantly \(1\).  If \(c,c'\) have label \(\Mt\), the failure of
\(\mathsf{Diverse}\) gives
\(\prof_{\C}(c)=\prof_{\C}(c')\).  For every \(e\in C\),
\[
  \prof_{\C}(e)
  =\{(v_e(T),T)\mid T\in\mathcal L_{\C}\}.
\]
It follows that \(v_c=v_{c'}\).  Thus \(v_c\) depends only on the label of
\(c\).

For \(S,T\in\mathcal L_{\C}\), choose any \(c\) of label \(S\) and define
\(\mu_{\C}(S,T):=v_c(T)\).  The preceding paragraph shows that this is
well defined, and the definition of \(v_c(T)\) gives
\eqref{eq:label-matrix}.  Hence \textup{(ii)} holds.

Conversely, assume \textup{(ii)} and let \(c\) have label \(S\).  Then
\eqref{eq:label-matrix} gives
\[
  \prof_{\C}(c)
  =\{(\mu_{\C}(S,T),T)\mid T\in\mathcal L_{\C}\}.
\]
This profile contains exactly one of \((0,T)\) and \((1,T)\) for every
realised label \(T\), and neither pair for an unrealised label.  Hence
\(\mathsf{Split}\) is false.  Moreover, all elements of label \(\Mt\) have
the same profile, so \(\mathsf{Diverse}\) is false.  Therefore
\(\C\models\neg\Phi\), proving \textup{(i)}.
\end{proof}

Thus every model of \(\neg\Phi\) is a blow-up of a possibly looped directed
graph on at most three vertices, indexed by the realised labels
\(\Zt,\Ot,\Mt\): each block of its \(R\)-matrix is constant.  The next lemma
shows that this form survives restriction, because the \(\Zt\)- and
\(\Ot\)-classes are fixed while the old \(\Mt\)-class either remains mixed or
changes label as a whole.

\begin{lem}
\label{lem:hereditary}
The class of structures satisfying \(\neg\Phi\) is closed under induced substructures.
\end{lem}

\begin{proof}
Let \(\C\subseteq\D\) and assume \(\D\models\neg\Phi\).  By
\cref{lem:label-matrix}, fix a label matrix \(\mu_{\D}\) for \(\D\).  We
need to construct a label matrix for \(\C\).  We first determine how
restriction can change the labels; this will determine the required matrix.

If \(c\in C\) has label \(\Zt\) in \(\D\), its row remains constantly
zero in \(\C\), so it retains label \(\Zt\).  Likewise, every element of
label \(\Ot\) in \(\D\) retains label \(\Ot\) in \(\C\).  Put
\[
  C_{\Mt}^{\D}:=\{c\in C\mid \lambda_{\D}(c)=\Mt\}.
\]
Since \(\C\) is induced in \(\D\), for all
\(c,c'\in C_{\Mt}^{\D}\) and \(d\in C\),
\[
  \rho_c^{\C}(d)
  =\rho_c^{\D}(d)
  =\mu_{\D}\bigl(\Mt,\lambda_{\D}(d)\bigr)
  =\rho_{c'}^{\D}(d)
  =\rho_{c'}^{\C}(d).
\]
Thus all elements of \(C_{\Mt}^{\D}\) have the same row in \(\C\) and, when
\(C_{\Mt}^{\D}\neq\varnothing\), acquire a common label
\(T_{\Mt}\in\Types\).

We now construct the required matrix in two cases.  If
\(C_{\Mt}^{\D}=\varnothing\), or if
\(C_{\Mt}^{\D}\neq\varnothing\) and
\(T_{\Mt}\in\{\Zt,\Ot\}\), then \(\C\) has no element of
label \(\Mt\).  Every row is therefore constantly zero or constantly one,
according to its source label.  Hence the function
\[
  \mu_{\C}(U,V):=
  \begin{cases}
    0,&U=\Zt,\\
    1,&U=\Ot
  \end{cases}
  \qquad(U,V\in\mathcal L_{\C})
\]
satisfies~\eqref{eq:label-matrix}, and
\cref{lem:label-matrix} gives \(\C\models\neg\Phi\).

If \(C_{\Mt}^{\D}\neq\varnothing\) and \(T_{\Mt}=\Mt\), every element has
the same label in \(\C\) as in \(\D\).  Consequently
\(\mathcal L_{\C}\subseteq\mathcal L_{\D}\), and
\[
  \mu_{\C}:=
  \mu_{\D}|_{\mathcal L_{\C}\times\mathcal L_{\C}}
\]
satisfies~\eqref{eq:label-matrix} for \(\C\).  Again
\cref{lem:label-matrix} yields \(\C\models\neg\Phi\).  The two cases exhaust
all possibilities.
\end{proof}

\begin{thm}
\label{thm:extension-preserved}
The sentence \(\Phi\) is preserved under extensions.
\end{thm}

\begin{proof}
We prove the claim by contraposition.  If \(\C\subseteq\D\) and
\(\D\models\neg\Phi\), then \cref{lem:hereditary} gives
\(\C\models\neg\Phi\).  Hence \(\C\models\Phi\) implies
\(\D\models\Phi\).
\end{proof}

\subsection{Existential fluted formulas do not define \texorpdfstring{\(\Phi\)}{Phi}}
\label{sec:los-tarski-separation}

We now apply the simulation principle from
\cref{prop:existential-transfer}.  We give two finite structures which will
witness the failure of existential definability: \(\A\) will satisfy
\(\Phi\), while \(\B\) will not, but existential fluted formulas will
transfer from \(\A\) to \(\B\) along rank-bounded fluted simulations.  Let
\[
  A=\{p,q,r\},
  \qquad
  R^{\A}=\{(p,q)\},
\]
and
\[
  B=\{u,v\},
  \qquad
  R^{\B}=\{(u,v)\}.
\]
The corresponding \(R\)-matrices are
\[
\begin{array}{cccc}
 R^{\A} & p&q&r\\ \midrule
 p&0&1&0\\
 q&0&0&0\\
 r&0&0&0
\end{array}
\qquad\qquad
\begin{array}{ccc}
 R^{\B} & u&v\\ \midrule
 u&0&1\\
 v&0&0
\end{array}.
\]

The separating pair isolates the information that these simulations are
allowed to forget.  In \(\A\), the element \(p\) distinguishes \(q\) and
\(r\), although \(q,r\) have the same row label \(\Zt\): indeed, they have
the same zero row.  In \(\B\), the corresponding positive and negative
continuations from \(u\) lead to elements of different row labels:
\(R(u,v)\) with \(v\) of label \(\Zt\), and \(\neg R(u,u)\) with \(u\) of
label \(\Mt\).  Thus \(\B\) reproduces the local \(R/\neg R\) choices
available from \(p\), but not the fact that the two choices in \(\A\) end in
the same row class.

The point is not that existential fluted formulas cannot detect both an
\(R\)-successor and a non-\(R\)-successor from one source; they can.  What
they cannot certify in this example is that the two witnesses both have zero
rows, since the formula defining the label \(\Zt\) uses universal
information.

\begin{lem}
\label{lem:Phi-values}
We have \(\A\models\Phi\) and \(\B\not\models\Phi\).
\end{lem}

\begin{proof}
We first verify the positive instance.  In \(\A\), we have \(\A\models\Zt(x_2)[q]\) and \(\A\models\Zt(x_2)[r]\), while
\[
  R^{\A}(p,q)
  \quad\text{and}\quad
  \neg R^{\A}(p,r).
\]
Thus \(\A\models\mathsf{Split}\).

We next verify the negative instance.  In \(\B\), the element \(u\) has
label \(\Mt\), the element \(v\) has label \(\Zt\), and the label \(\Ot\)
is not realised.  Their profiles are
\[
  \prof_{\B}(u)=\{(0,\Mt),(1,\Zt)\},
  \qquad
  \prof_{\B}(v)=\{(0,\Mt),(0,\Zt)\}.
\]
Each profile contains at most one of \((0,T)\) and \((1,T)\) for every
label \(T\), so \(\mathsf{Split}\) is false.  Moreover, \(u\) is the only
element of label \(\Mt\), so \(\mathsf{Diverse}\) is false.  Hence
\(\B\not\models\Phi\).
\end{proof}

To transfer existential fluted sentences of arbitrary rank, it remains to
construct a rank-$k$ fluted simulation from \(\A\) to \(\B\) for every
$k\geq0$.  Let
\[
  S:=\{(p,u),(q,u),(r,u),(q,v),(r,v)\}\subseteq A\times B.
\]
The map \(\iota\) below handles the first move of the simulation:
\[
  \iota(p)=u,
  \qquad
  \iota(q)=\iota(r)=v.
\]
After a non-empty history, the following table handles each subsequent move.
For a current pair \((a,b)\in S\) and a next element \(a'\in A\), it specifies
the matching element of \(B\):
\begin{equation}
\begin{array}{cccc}
  \delta((a,b),a')&a'=p&a'=q&a'=r\\ \midrule
  (p,u)&u&v&u\\
  (q,u)&u&u&u\\
  (r,u)&u&u&u\\
  (q,v)&u&v&v\\
  (r,v)&u&v&v
\end{array}
\label{eq:transition-table}
\end{equation}
Denote the table entry by \(\delta((a,b),a')\).

The first row of~\eqref{eq:transition-table} contains the essential choices.
After \(p\) has been matched with \(u\), the continuation \(q\) is matched
with \(v\) to preserve the edge, whereas \(r\) is matched with \(u\) to
preserve the non-edge.  Thus the two \(\Zt\)-elements \(q,r\) need not be
matched with elements of the same row label in \(\B\).  The remaining rows
ensure that this history-dependent matching can continue indefinitely.

\begin{lem}
\label{lem:transition}
For every \((a,b)\in S\) and \(a'\in A\), if
\[
  b':=\delta((a,b),a'),
\]
then \((a',b')\in S\) and
\begin{equation}
  R^{\A}(a,a')
  \quad\Longleftrightarrow\quad
  R^{\B}(b,b').
  \label{eq:edge-match}
\end{equation}
\end{lem}

\begin{proof}
We verify closure in \(S\) and edge matching separately.  For each of the
five possible current pairs in \(S\), the entries in the corresponding row
of~\eqref{eq:transition-table} give pairs \((a',b')\) that again belong to
\(S\):
\[
\begin{array}{cccc}
  (a,b)&a'=p&a'=q&a'=r\\ \midrule
  (p,u)&(p,u)&(q,v)&(r,u)\\
  (q,u)&(p,u)&(q,u)&(r,u)\\
  (r,u)&(p,u)&(q,u)&(r,u)\\
  (q,v)&(p,u)&(q,v)&(r,v)\\
  (r,v)&(p,u)&(q,v)&(r,v).
\end{array}
\]
This proves the membership assertion.

For edge matching, recall that \((p,q)\) is the only pair in \(R^\A\) and
\((u,v)\) is the only pair in \(R^\B\).  The table produces
\((b,b')=(u,v)\) exactly in the row \((a,b)=(p,u)\) and column \(a'=q\),
that is, exactly when \((a,a')=(p,q)\).  This proves
\eqref{eq:edge-match}.
\end{proof}

We now use the transition table to build a relation on corresponding finite
histories.
Let
\begin{align*}
 Z:={}&\{(\epsilon,\epsilon)\}\\
 &{}\cup\bigl\{((a_1,\ldots,a_m),(b_1,\ldots,b_m))\mid m\geq1,
       \ (a_i,b_i)\in S\text{ for every }i,\\
 &\hspace{18mm}
       R^{\A}(a_i,a_{i+1})\leftrightarrow R^{\B}(b_i,b_{i+1})
       \text{ for every }i<m\bigr\}.
\end{align*}
Thus membership in $Z$ records exactly the atomic information visible along
the consecutive pairs of a fluted history.

\begin{lem}
\label{lem:concrete-simulation}
For every $k\geq0$, setting $Z_j:=Z$ for $0\leq j\leq k$ yields a rank-$k$
fluted simulation from \(\A\) to \(\B\).  In particular,
$(\A,\epsilon)\preccurlyeq_k^{\FL}(\B,\epsilon)$.
\end{lem}

\begin{proof}
We verify atomic harmony, the forth condition, and the root pair.  There are
no $R$-atoms at levels $0$ and $1$.
If $(\bar a,\bar b)\in Z$ has length $m\geq2$, the only current fluted atom
over $\sigma$ is $R(x_{m-1},x_m)$.  The last condition in the definition of
$Z$, with $i=m-1$, gives
\[
  R^{\A}(a_{m-1},a_m)
  \quad\Longleftrightarrow\quad
  R^{\B}(b_{m-1},b_m),
\]
which is precisely atomic harmony.

It remains to prove the forth condition.  Let $(\bar a,\bar b)\in Z$ and choose $c\in A$.
If $\bar a=\epsilon$, put $d:=\iota(c)$.  Then
$(c,d)\in S$, and no adjacency condition is required for a one-element
history; hence $((c),(d))\in Z$.

Suppose now that $\bar a$ is non-empty, and let $(a',b')$ be the last
coordinate pair of $(\bar a,\bar b)$.  Membership in $Z$ gives
$(a',b')\in S$.  Put
\[
  d:=\delta((a',b'),c).
\]
By \cref{lem:transition}, $(c,d)\in S$ and
$R^{\A}(a',c)\leftrightarrow R^{\B}(b',d)$.  Every earlier coordinate pair
and every earlier adjacency condition is inherited from $(\bar a,\bar b)$.
Therefore $(\bar a c,\bar b d)\in Z$.  Thus the constant sequence
$(Z_0,\ldots,Z_k)$ satisfies atomic harmony and the forth condition for every
$k\geq0$.
Moreover, $(\epsilon,\epsilon)\in Z=Z_k$, so $Z_k$ is non-empty and
$(\A,\epsilon)\preccurlyeq_k^{\FL}(\B,\epsilon)$.
\end{proof}

\begin{thm}
\label{thm:los-tarski-failure}
The rank-three sentence \(\Phi\in\FL[0](\sigma)\) is preserved under
extensions but has no existential fluted equivalent over all structures.
Consequently, the fragment-internal \L o\'s--Tarski theorem fails already for
the single-binary-relation signature \(\sigma=\{R\}\).
\end{thm}

\begin{proof}
Membership and quantifier rank follow from \cref{prop:phi-syntax}, and
extension preservation follows from \cref{thm:extension-preserved}.  It
remains to rule out existential fluted definability.  Suppose, towards a
contradiction, that an existential fluted sentence \(\psi\) is equivalent to
\(\Phi\), and put $k:=\qr(\psi)$.  By \cref{lem:Phi-values},
\(\A\models\Phi\), and hence \(\A\models\psi\).  The rank-$k$ simulation
from \cref{lem:concrete-simulation} and existential transfer
\cref{prop:existential-transfer} then give \(\B\models\psi\).  The assumed
equivalence would imply \(\B\models\Phi\), contradicting
\cref{lem:Phi-values}.
\end{proof}

\begin{cor}
\label{cor:finite-los-tarski}
Over finite structures, \(\Phi\) is preserved under extensions but has no
existential fluted equivalent.
\end{cor}

\begin{proof}
Preservation follows from \cref{thm:extension-preserved}.  The structures
\(\A\) and \(\B\) are finite, and
\cref{lem:concrete-simulation,prop:existential-transfer} show that every
existential fluted sentence true in \(\A\) is true in \(\B\), whereas
\cref{lem:Phi-values} gives \(\A\models\Phi\) and \(\B\not\models\Phi\).
Thus no existential fluted sentence is equivalent to \(\Phi\) over finite
structures.
\end{proof}

\section{Failure of weak Beth definability}\label{sec:wbdp}
This section refutes weak Beth definability by giving a strong implicit
definition of a binary relation $R$ over the unary base signature $\{P\}$.
The unique expansion interprets $R$ as $P^{\A}\times P^{\A}$.  However, no
formula in $\FL[2](\{P\})$ can define this relation, because over a unary
signature its truth depends only on the final coordinate.  A two-element
base structure witnesses non-definability both over all structures and over
finite structures.

We begin by fixing the notions of implicit and explicit fluted definability,
including the distinction between weak and strong implicit definitions.

\begin{defi}[Implicit and explicit fluted definitions]\label{def:wbdp}
Let $\tau$ be a base signature, let $S\notin\tau$ be an $r$-ary relation
symbol, let $T\subseteq\FL[0](\tau)$, and let
\[
   \Gamma(S)\subseteq\FL[0](\tau\cup\{S\}).
\]
\begin{enumerate}[label=\textup{(\roman*)}]
  \item $\Gamma(S)$ is a \emph{weak implicit definition} of $S$ relative to
  $T$ if, for every $\tau$-structure $\A\models T$, there is at most one
  relation $S^{\ast}\subseteq A^r$ such that
  \[
      (\A,S^{\ast})\models\Gamma(S).
  \]
  It is a \emph{strong implicit definition} if, for every
  $\tau$-structure $\A\models T$, there is exactly one such relation.
  \item $S$ is \emph{explicitly fluted-definable} relative to
  $T\cup\Gamma(S)$ if there is a formula
  \[
      \delta\in\FL[r](\tau)
  \]
  such that every $(\A,S^{\ast})\models T\cup\Gamma(S)$ satisfies, for all $\bar a\in A^r$,
  \[
       S^{\ast}(\bar a)
       \quad\Longleftrightarrow\quad
       \A\models\delta[\bar a].
  \]
\end{enumerate}
The fragment has the \emph{Beth definability property} (also called the
\emph{strong} or \emph{ordinary Beth definability property}) if, for every
choice of $\tau,S,T$, and $\Gamma(S)$ as above, every weak implicit definition
has an explicit fluted definition.  It has the \emph{weak Beth definability
property} if, for every such choice, every strong implicit definition has an
explicit fluted definition.
\end{defi}

Following the terminology of Andr\'eka and
N\'emeti~\cite[Definition~2]{AndrekaNemeti2021}, the weak Beth property requires explicit
definability only for strong implicit definitions, which have both existence
and uniqueness.  The ordinary Beth property also covers weak implicit
definitions, which require uniqueness alone.  Thus the ordinary Beth property
implies the weak Beth property.  In the finite versions, implicit
definability is required only over finite base structures and explicit
definability only over finite models of $T\cup\Gamma(S)$.

To refute both properties, it therefore suffices to give a strong implicit
definition with no explicit fluted definition.  We now construct one over the
unary base signature $\{P\}$.

Let the base signature be $\tau=\{P\}$ with $P$ unary, and let $R$ be a new
binary symbol.  Put
\begin{align}
 \Sigma(P,R):={}&\forall x_1\Bigl[
   \bigl(P(x_1)\rightarrow
      \forall x_2\bigl(R(x_1,x_2)\leftrightarrow P(x_2)\bigr)\bigr)
   \notag\\[-1mm]
   &\hspace{34mm}\wedge
   \bigl(\neg P(x_1)\rightarrow
      \forall x_2\neg R(x_1,x_2)\bigr)
 \Bigr].
 \label{eq:sigma-beth}
\end{align}

\begin{lem}\label{lem:sigma-fluted}
The sentence $\Sigma(P,R)$ belongs to $\FL[0](\{P,R\})$, uses only
$x_1,x_2$, and has quantifier rank two.
\end{lem}

\begin{proof}
We verify the fluted levels and quantifier rank from the inside out.  At level
$2$, both
$R(x_1,x_2)$ and $P(x_2)$ are permitted suffix atoms.  Hence
\[
  R(x_1,x_2)\leftrightarrow P(x_2)
  \quad\text{and}\quad
  \neg R(x_1,x_2)
\]
belong to $\FL[2](\{P,R\})$; expanding the biconditional uses only Boolean
connectives.  Quantifying $x_2$ puts the two resulting formulas in
$\FL[1](\{P,R\})$.  At that level they can be combined with the permitted
atom $P(x_1)$ and its negation to form the two implications in
\eqref{eq:sigma-beth}.  Finally, quantifying $x_1$ gives a sentence in
$\FL[0](\{P,R\})$.  Each branch has two nested quantifiers, so the quantifier
rank is two.
\end{proof}

\begin{prop}[Existence and uniqueness]\label{prop:unique-expansion}
Let $\A$ be any structure over $\{P\}$.  There is exactly one binary relation
$R^{\ast}$ such that
\[
  (\A,R^{\ast})\models\Sigma(P,R).
\]
It is
\[
   R^{\ast}=P^{\A}\times P^{\A}.
\]
\end{prop}

\begin{proof}
We first prove uniqueness and identify the only possible relation.  Suppose
$(\A,R^{\ast})\models\Sigma(P,R)$.  If $a\in P^{\A}$, the first implication in
\eqref{eq:sigma-beth} gives, for every $b\in A$,
\[
  R^{\ast}(a,b)\quad\Longleftrightarrow\quad b\in P^{\A}.
\]
If $a\notin P^{\A}$, the second implication gives
\[
  \neg R^{\ast}(a,b)
\]
for every $b\in A$.  Therefore
\[
  R^{\ast}(a,b)
  \quad\Longleftrightarrow\quad
  a\in P^{\A}\ \text{and}\ b\in P^{\A},
\]
so $R^{\ast}=P^{\A}\times P^{\A}$.

It remains to prove existence.  Set $R^{\ast}=P^{\A}\times P^{\A}$.  For
every $b\in A$, if $a\in P^{\A}$, then
$R^{\ast}(a,b)$ holds exactly when $b\in P^{\A}$; if
$a\notin P^{\A}$, then $R^{\ast}(a,b)$ is false.  Hence the expansion
$(\A,R^{\ast})$ satisfies $\Sigma(P,R)$.
\end{proof}

Thus the singleton theory $\{\Sigma(P,R)\}$ is a strong implicit definition
of $R$ relative to the empty base theory.  To rule out an explicit definition,
we next show that a fluted formula over a unary signature at level $m$ can
depend only on its final coordinate.

\begin{lem}[Last-coordinate invariance]\label{lem:last-coordinate}
Let $\tau$ be a unary relational signature, let $m\geq1$, and let
$\theta\in\FL[m](\tau)$.  For every $\tau$-structure $\A$ and all tuples
\[
   (a_1,\ldots,a_m),\ (b_1,\ldots,b_m)\in A^m
\]
with $a_m=b_m$, one has
\[
  \A\models\theta[a_1,\ldots,a_m]
  \quad\Longleftrightarrow\quad
  \A\models\theta[b_1,\ldots,b_m].
\]
In other words, the truth of a fluted formula over a unary signature at level
$m$ depends only on its last coordinate.
\end{lem}

\begin{proof}
Fix $\theta\in\FL[m](\tau)$ and put $q:=\qr(\theta)$.  We prove the claim by
constructing a rank-$q$ self-bisimulation of $\A$ that relates tuples with the
same last coordinate.  Define
\[
  Z:=\{(\epsilon,\epsilon)\}
     \cup
     \bigl\{((c_1,\ldots,c_\ell),(d_1,\ldots,d_\ell))\mid
       \ell\geq1\text{ and }c_\ell=d_\ell\bigr\}.
\]
For every $0\leq j\leq q$, put $Z_j:=Z$.

We first verify atomic harmony.  In a unary signature, every permitted atom
at level $\ell\geq1$ has the form $U(x_\ell)$.  For a pair in $Z$, the two last
coordinates are equal, so $U$ has the same truth value on both tuples.  At
level $0$ there are no atoms.

For the forth and back conditions, given any pair in $Z$ and any $e\in A$,
append $e$ on both sides.  The resulting tuples again have the same last
coordinate and therefore form a pair in $Z$.  Thus
$(Z_0,\ldots,Z_q)$ is a rank-$q$ fluted bisimulation between $\A$ and itself.
Since $(\epsilon,\epsilon)\in Z=Z_q$, its top relation is non-empty.

Finally, the tuples in the statement have the same last coordinate, so their
pair belongs to $Z_q$.  Rank-bounded bisimulation invariance
(\cref{prop:bisimulation-invariance}) gives the required equivalence for
$\theta$.
\end{proof}

\begin{thm}\label{thm:wbdp-proof}
Over all structures, the relation symbol $R$ has no explicit fluted definition
relative to the singleton theory $\{\Sigma(P,R)\}$.
\end{thm}

\begin{proof}
We now combine the unique expansion with last-coordinate invariance.  Assume,
towards a contradiction, that some
\[
  \delta\in\FL[2](\{P\})
\]
explicitly defines $R$ in all models of $\Sigma(P,R)$.

To force a contradiction, we choose a base structure in which the uniquely
defined relation distinguishes two tuples having the same second coordinate.
Let $\A$ have domain $A=\{0,1\}$ and interpret
\[
   P^{\A}=\{1\}.
\]
By \cref{prop:unique-expansion}, the unique expansion satisfying
$\Sigma(P,R)$ has
\[
   R^{\ast}=\{(1,1)\}.
\]
Consequently,
\[
   R^{\ast}(0,1)\ \text{is false},
   \qquad
   R^{\ast}(1,1)\ \text{is true}.
\]
Explicit definability would therefore require
\[
  \A\not\models\delta[0,1]
  \qquad\text{and}\qquad
  \A\models\delta[1,1].
\]
But the tuples $(0,1)$ and $(1,1)$ have the same last coordinate.  By
\cref{lem:last-coordinate},
\[
  \A\models\delta[0,1]
  \quad\Longleftrightarrow\quad
  \A\models\delta[1,1],
\]
contradicting the assumption that $\delta$ defines $R^{\ast}$.
\end{proof}

\begin{cor}\label{cor:wbdp-finite}
Both the weak and ordinary Beth definability properties fail for $\FL$, over
all structures and over finite structures.
\end{cor}

\begin{proof}
The same witness works in both semantics.  By
\cref{lem:sigma-fluted,prop:unique-expansion}, $\{\Sigma(P,R)\}$ is a strong
implicit definition over all structures, and hence over finite structures.
The proof of \cref{thm:wbdp-proof} uses a finite two-element base structure,
and \cref{lem:last-coordinate} gives the same contradiction when only finite
models are considered.  Thus the weak Beth definability property fails over
both classes of structures.

In either semantics, every strong implicit definition is also weak, since
``exactly one'' entails ``at most one''.  The same witness therefore refutes
the ordinary Beth definability property.
\end{proof}

\section{Conclusion}\label{sec:conclusion}

Our results show that the fluted fragment supports an equirank homomorphism
preservation theorem, whereas the fragment-internal \L o\'s--Tarski theorem
and both Beth definability properties fail.  For homomorphisms,
existential-positive characteristic sentences unfold into finite canonical
forests, and the reset lemma yields an equirank equivalent.  The
counterexamples expose two different limitations: a fluted simulation need
not preserve the fact that two continuations end in the same row-label class,
and the truth of a level-two formula over a unary signature depends only on
its second coordinate.  All results persist over finite structures.  The
arguments are equality-free; adding equality invalidates the reset and
last-coordinate lemmas and lies outside the present scope.

\section*{Acknowledgements}
This work was supported by the China Scholarship Council (grant no.~202206220057). The author thanks Hongkai Yin for useful discussions on the fluted fragment.

\bibliographystyle{alphaurl}
\bibliography{references}
\end{document}